\documentclass{amsart}

\usepackage{winart}

\title{An equivalence linking CM-types $A_\infty$ and $D_\infty$}
\date{\today}



\author[C. Cummings]{Charley Cummings}
\address{Charley Cummings, Department of Mathematics, Aarhus University, Ny Munkegade 118, 8000 Aarhus C, Denmark}
\email{c.cummings@math.au.dk}

\author[S. Gratz]{Sira Gratz}
\address{Sira Gratz, Department of Mathematics, Aarhus University, Ny Munkegade 118, 8000 Aarhus C, Denmark}
\email{sira@math.au.dk}

\author[E. Kirkman]{Ellen Kirkman}
\address{Ellen Kirkman, Department of Mathematics, Wake Forest University, Winston-Salem, NC, USA}
\email{kirkman@wfu.edu}

\author[J.~C.~Letz]{Janina C. Letz}
\address{Janina~C.~Letz,
Institute of Mathematics,
Paderborn University,
Warburger Str. 100,
33098 Paderborn,
Germany
}
\email{jletz@math.upb.de}

\author[J. D. Rock]{J. Daisie Rock}
\address{J. Daisie Rock, Section of Algebra, Department of Mathematics, KU Leuven, Celestijnenlaan 200B bus 2400,
B-3001 Leuven,
Belgium \newline
\indent \indent Department of Mathematics W16, Ghent University, 9000 Ghent, East Flanders, Belgium}
\email{jobdaisie.rock@kuleuven.be}

\author[Š. Špenko]{Špela Špenko}
\address{Špela Špenko, Département de Mathématique, Université Libre de Bruxelles, Campus de la
Plaine CP 213, Bld du Triomphe, B-1050 Bruxelles, Belgium}
\email{spela.spenko@ulb.be}

\begin{document}

\begin{abstract}
We show that, for a specific grading, the stable categories of graded maximal Cohen-Macaulay modules over hypersurfaces of type $A_\infty$ and $D_\infty$ are equivalent.
\end{abstract}

\maketitle


\section{Introduction}

There are two hypersurface singularities of countably infinite Cohen-Macaulay type of embedding dimension $2$: type $A_\infty$ and type $D_\infty$. We consider these singularities in the following graded setting: Let $k$ be an uncountable algebraically closed field of characteristic different from $2$ and let $S = k[x,y]$ be the $\BZ$-graded polynomial ring with $x$ in degree $1$ and $y$ in degree $-1$. The (stable) category of graded maximal Cohen-Macaulay modules of the type $A_\infty$ hypersurface singularity $S/\langle x^2\rangle$ has been extensively studied from a cluster-theoretic perspective in \cite{Fisher:2017,Paquette/Yildirim:2021,August/Cheung/Faber/Gratz/Schroll:2023,August/Cheung/Faber/Gratz/Schroll:2024,Cummings/Gratz:2024,Canakci/Kalck/Pressland:2025}. We consider that same grading for the type $D_\infty$ hypersurface singularity $S/\langle x^2y \rangle$. 

\begin{introthm}\label{T:main_thm}
	We have a triangular equivalence
\[
	\smcm[\BZ]{S/\langle x^2 \rangle} \cong \smcm[\BZ]{S/\langle x^2 y \rangle}\,,
\]
where $\smcm[\BZ]{R}$ is the stable category of $\BZ$-graded maximal Cohen-Macaulay modules over $R$.
\end{introthm}

For a Gorenstein ring, the stable category of maximal Cohen-Macaulay modules is equivalent to the singularity category; see \cite[Theorem~4.4.1]{Buchweitz:2021}. Triangular equivalences between singularity categories are very rare. Apart from Knörrer periodicity \cite{Knorrer:1987,Solberg:1989}, in the case of Gorenstein singularities no examples exist. However, Kalck \cite{Kalck:2021} recently constructed  equivalences for specific non-Gorenstein quotient singularities. In the case of graded singularity categories, there is an analogue of Knörrer periodicity for certain GL complete intersections  \cite[Corollary 3.23]{Herschend/Iyama/Minamoto/Oppermann:2023}. We do not know of any other equivalences. 

In some cases it has been established that there are simply no equivalences besides Kn\"orrer periodicity: For two complete local singularities over $\BC$, where one is an isolated hypersurface singularity or an isolated Gorenstein singularity, Kalck shows that if the dg singularity categories are quasi-equivalent, then the rings are related via Knörrer periodicity; see \cite{Kalck:2021a}. Let us also mention that for excellent local Gorenstein hypersurfaces, an equivalence of singularity categories implies a homeomorphism of singular loci; see \cite[Theorem 4.4]{Matsui:2019}. 


\begin{ack}
This work started during the WINART4 workshop. We are thankful to the organisers of that meeting. 
C.C.\@ and S.G.\@ was supported by VILLUM FONDEN (Grant Number VIL42076). 
J.C.L.\@ was partly supported by the Deutsche Forschungsgemeinschaft (SFB-TRR 358/1 2023 - 491392403).
J.D.R.\@ is supported by FWO grant 1298325N and partially supported by the FWO grants G0F5921N (Odysseus) and G023721N, and by the KU Leuven grant iBOF/23/064.
Š.Š.\@ was supported by an ARC grant from the Université libre de Bruxelles. 
\end{ack}

\section{Stable category of maximal Cohen-Macaulay modules}

Let $R$ be a noetherian Gorenstein ring. We denote the category of maximal Cohen-Macaulay modules over $R$ by $\mcm{R}$ and its stabilisation by $\smcm{R}$. The stable category of maximal Cohen-Macaulay modules over a Gorenstein ring $R$ is equivalent to several other categories. It is triangular equivalent to the homotopy category of acyclic complexes of finitely generated projective modules over $R$; see \cite[Theorem~4.4.1]{Buchweitz:2021}. This equivalence provides a convenient framework for computing morphisms in the stable category. Namely, by \cite[Lemma~6.1.2]{Buchweitz:2021}, we have
\begin{equation} \label{shom_hh_CR}
\Hom{\smcm{R}}{M}{N} = \hh{0} \Hom{R}{\CR(M)}{N}
\end{equation}
when $M$ and $N$ are maximal Cohen-Macaulay modules and $\CR(M)$ is a complete resolution of $M$; that is an acyclic complex of finitely generated projective $R$-modules with $\coker(d^{-1}) = M$.


For a hypersurface $R=S/\langle f\rangle$ of a polynomial ring $S$ and $f \in S$, there is a correspondence of the maximal Cohen-Macaulay $R$-modules and matrix factorizations of $f$. A matrix factorisation $(\Phi,\Psi)$ consists of maps of free $S$-modules
\[
F\xrightarrow{\Psi} G\xrightarrow{\Phi} F
\]
such that $\Phi \Psi = f \id_F$ and $\Psi \Phi = f \id_G$. We denote the category of matrix factorisations of $f$ by $\mf{S}{f}$ and its stabilisation by $\smf{S}{f}$; for more details see \cite[Chapter~7]{Yoshino:1990}. There is a triangular equivalence 
\begin{equation*}
\smf{S}{f} \to \smcm{S/\langle f \rangle} \quad \text{given by} \quad (\Phi,\Psi) \mapsto \coker(\Phi)\,;
\end{equation*}
see \cite{Eisenbud:1980} and \cite{Orlov:2004}.
Moreover, a matrix factorisation $(\Phi,\Psi)$ induces a complete resolution of $\coker(\Phi)$ given by
\begin{equation*}
\cdots \to \bar{F} \xrightarrow{\bar{\Psi}} \bar{G} \xrightarrow{\bar{\Phi}} \bar{F} \to \cdots
\end{equation*}
where $\bar{(-)} \coloneqq S/\langle f \rangle \otimes_S -$. From this description it is clear that the syzygy functor on $\smcm{S/\langle f \rangle}$ corresponds to the inverse suspension functor on the category of acyclic complexes of finitely generated projective $S/\langle f \rangle$-modules.


\section{MCM-modules over a \texorpdfstring{$D_\infty$}{D-infinity}-singularity}

Let $k$ be an uncountable algebraically closed field of characteristic not 2 and $S = k[x,y]$ be a $\BZ$-graded polynomial ring with $|x| = 1$ and $|y| = -1$. We consider the hypersurface $R = S/\langle x^2y \rangle$. By \cite[Theorem~B]{Buchweitz/Greuel/Schreyer:1987}, the ring $R$ has countable CM-type; also see \cite[Proposition~14.19]{Leuschke/Wiegand:2012}. The indecomposable maximal Cohen-Macaulay $R$-modules were classified in \cite[Proposition~4.2]{Buchweitz/Greuel/Schreyer:1987}. The indecomposable graded maximal Cohen-Macaulay $R$-modules and their corresponding matrix factorizations are depicted in \cref{indec_mcm}. The indices of the modules are chosen so that the generators lie in degree $i$ (and $j$). 

\begin{figure}
\centering

\includegraphics[scale=0.8]{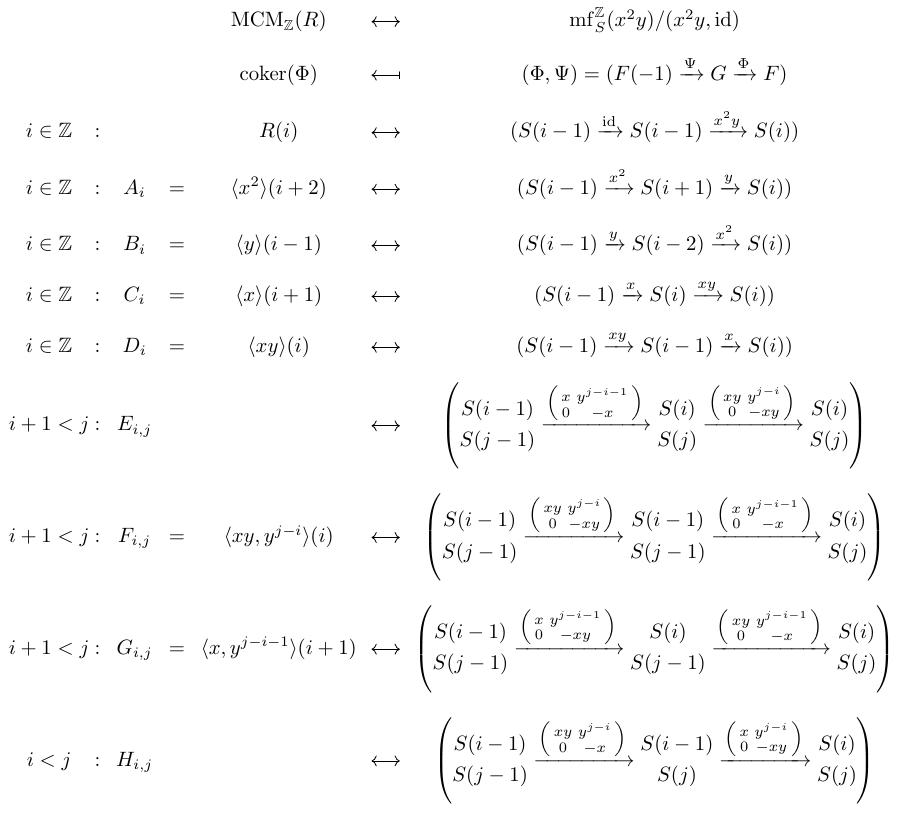}
\caption{The indecomposable graded maximal Cohen-Macaulay modules over $R = k[x,y]/\langle x^2y \rangle$ with the grading $|x| = 1$ and $|y| = -1$, and their corresponding matrix factorizations.}
\label{indec_mcm}
\end{figure}

While the modules $E_{i,i+1}$, $F_{i,i+1}$, $G_{i,i+1}$ and $H_{i,i}$ are well-defined, they do not appear in the list, as they are not `new'  maximal Cohen-Macaulay modules. In fact, we have
\begin{equation*}
E_{i,i+1} = A_i \oplus R(i+1) \,,\quad F_{i,i+1} = B_{i+1} \,,\quad G_{i,i+1} = R(i+1) \quad \text{and} \quad H_{i,i} = R(i)\,.
\end{equation*}
This means, that in the stable category of  maximal Cohen-Macaulay modules we can replace $A_i$ and $B_i$ by $E_{i,i+1}$ and $F_{i-1,i}$, respectively. 


From the matrix factorisations it is easy to see
\begin{gather} \label{MCM_syzygy}
\begin{aligned}
\Omega A_i &= B_{i+1} \,, & \Omega B_i &= A_{i-2} \\
\Omega C_i &= D_i \,, & \Omega D_i &= C_{i-1} \\
\Omega E_{i,j} &= F_{i,j} \,, & \Omega F_{i,j} &= E_{i-1,j-1} \\
\Omega G_{i,j} &= H_{i,j-1} \,, & \Omega H_{i,j} &= G_{i-1,j}\,.
\end{aligned}
\end{gather}

\begin{lemma}\label{L:generates}
	The object $C_0$ classically generates ${\smcm[\BZ]{R}}$, i.e.\ the thick closure of $C_0$ is ${\smcm[\BZ]{R}}$.
\end{lemma}

\begin{proof}
For every $i \in \BZ$, both $C_i$ and $D_i$ can be obtained as an appropriate suspension or desuspension of $C_0$. Moreover, for all $i,j \in \BZ$ we have the following short exact sequences in $\mcm[\BZ]{R}$:
\begin{equation*}
\begin{tikzcd}[ampersand replacement=\&,row sep=0]
0 \ar[r] \& D_i \ar[r,"{\begin{psmallmatrix} 1 \\ 0 \end{psmallmatrix}}"] \& E_{i,j} \ar[r,"{\begin{psmallmatrix} 0 & 1 \end{psmallmatrix}}"] \& D_j \ar[r] \& 0 \& \text{for} \& j \geq i+1 \\
0 \ar[r] \& C_i \ar[r,"{\begin{psmallmatrix} 1 \\ 0 \end{psmallmatrix}}"] \& F_{i,j} \ar[r,"{\begin{psmallmatrix} 0 & 1 \end{psmallmatrix}}"] \& C_j \ar[r] \& 0 \& \text{for} \&[-3em] j \geq i+1 \\
0 \ar[r] \& C_i \ar[r,"{\begin{psmallmatrix} 1 \\ 0 \end{psmallmatrix}}"] \& G_{i,j} \ar[r,"{\begin{psmallmatrix} 0 & 1 \end{psmallmatrix}}"] \& D_j \ar[r] \& 0 \& \text{for} \& j \geq i+1 \\
0 \ar[r] \& D_i \ar[r,"{\begin{psmallmatrix} 1 \\ 0 \end{psmallmatrix}}"] \& H_{i,j} \ar[r,"{\begin{psmallmatrix} 0 & 1 \end{psmallmatrix}}"] \& C_j \ar[r] \& 0 \& \text{for} \& j \geq i\,.
\end{tikzcd}
\end{equation*}
These short exact sequences induce exact triangles in $\smcm[\BZ]{R}$. Hence all indecomposable maximal Cohen-Macaulay modules lie in the thick closure of $C_0$, and $C_0$ classically generates $\smcm[\BZ]{R}$.
\end{proof}

Since $C_0$ is a classical generator of $\smcm[\BZ]{R}$, to understand the category $\smcm[\BZ]{R}$, it is enough to understand the graded endomorphism ring of $C_0$. 

\begin{lemma}\label{L:endoring}
	The graded endomorphism ring 
\begin{equation*}
\sExt[*]{R}{C_0}{C_0} = \bigoplus_{n \in \BZ} \Hom{\smcm[\BZ]{R}}{C_0}{\Omega^{-n} C_0}
\end{equation*}
is isomorphic to $k[t]$ as a graded algebra with $t$ in degree $-1$.
\end{lemma}

\begin{proof}
We compute the spaces $\Hom{\smcm[\BZ]{R}}{C_i}{C_k}$ and $\Hom{\smcm[\BZ]{R}}{C_i}{D_k}$ for all $i,k \in \BZ$ using \cref{shom_hh_CR}. As $C_k$ and $D_k$ are isomorphic to ideals, we use the following strategy:
\begin{enumerate}
\item find the cocycles in $\Hom{\grmod{R}}{\CR(C_i)}{R(m)}$; 
\item for each cocycle in $\Hom{\grmod{R}}{\CR(C_i)}{R(m)}$ identify all homotopies; 
\item for the ideals $C_k \subseteq R(k+1)$ and $D_k \subseteq R(k)$ identify the cocycles that restrict to morphisms to each of these ideals, while their homotopies do not restrict to the ideal.
\end{enumerate}

We consider the diagram
\begin{equation*}
\begin{tikzcd}
\CR(C_i) &[-3em] : &[-2em] \cdots \ar[r] & R(i) \ar[r,"xy"] \ar[dr,"0"] & R(i) \ar[d] \ar[r,"x"] & R(i+1) \ar[r] \ar[dl,dashed] & \cdots \nospacepunct{.} \\
& & & & R(m)
\end{tikzcd}
\end{equation*}
This yields the following cocycles with associated homotopies:
\begin{enumerate}[a)]
\item $x = \partial(1)$ when $m = i+1$; 
\item $x^{m-i} = \partial(x^{m-i-1})$ when $m > i+1$; 
\item $x y^{i-m+1} = \partial(y^{i-m+1})$ when $i+1 > m$. 
\end{enumerate}

We compute the morphisms with target $C_k = \langle x \rangle(k+1) \subseteq R(k+1)$; that is $m = k+1$. We obtain non-zero morphisms for:
\begin{enumerate}[a)]
\item $k+1 = i+1$; 
\item $k+1 > i+1$ and $k+1-i-1 < 1$; 
\item $i+1 > k+1$. 
\end{enumerate}
This yields
\begin{equation} \label{dimHom_C_C}
\dim_k \Hom{\smcm[\BZ]{R}}{C_i}{C_k} = \begin{cases}
1 & i \geq k \\
0 & \text{otherwise}\,.
\end{cases}
\end{equation}
A generator for $\Hom{\smcm[\BZ]{R}}{C_i}{C_k}$ is given by multiplication by $y^{i-k}$. We see this by considering the map of the respective complete resolutions:
\begin{equation}\label{E:CtoC}
\begin{tikzcd}
\cdots \ar[r] & R(i) \ar[r,"xy"] \ar[d,"{y^{i-k}}"] & R(i) \ar[d, "{y^{i-k}}"] \ar[r,"x"] & R(i+1) \ar[d, "{y^{i-k}}"] \ar[r] & \cdots \\
\cdots \ar[r] & R(k) \ar[r,"xy"] & R(k) \ar[r,"x"] & R(k+1) \ar[r] & \cdots \nospacepunct{.}
\end{tikzcd}
\end{equation}

We compute the morphisms with target $D_k = \langle xy \rangle(k) \subseteq R(k)$; that is $m = k$. We obtain non-zero morphisms for:
\begin{enumerate}[a)]
\item never well-defined;
\item never well-defined;
\item $i \geq k$. 
\end{enumerate}
This yields
\begin{equation} \label{dimHom_C_D}
\dim_k \Hom{\smcm[\BZ]{R}}{C_i}{D_k} = \begin{cases}
1 & i \geq k \\
0 & \text{otherwise}\,.
\end{cases}
\end{equation}
A generator for $\Hom{\smcm[\BZ]{R}}{C_i}{D_k}$ is given by multiplication by $y^{i-k}$. We see this by considering the map of the respective complete resolutions:
\begin{equation}\label{E:CtoD}
\begin{tikzcd}
\cdots \ar[r] & R(i) \ar[r,"xy"] \ar[d,"{y^{i-k+1}}"] & R(i) \ar[d, "{y^{i-k}}"] \ar[r,"x"] & R(i+1) \ar[d, "{y^{i-k+1}}"] \ar[r] & \cdots \\
\cdots \ar[r] & R(k-1) \ar[r,"x"] & R(k) \ar[r,"xy"] & R(k) \ar[r] & \cdots \nospacepunct{.}
\end{tikzcd}
\end{equation}


From \cref{dimHom_C_C,dimHom_C_D} we deduce that $k[t]$ and the graded endomorphism ring of $C_0$ are isomorphic as graded vector spaces. With \cref{E:CtoC,E:CtoD}, we obtain an isomorphism as graded algebras.
\end{proof}

Consider the hypersurface $R' =S/\langle x^2 \rangle$ of type $A_\infty$, where $S = k[x,y]$ is our $\BZ$-graded polynomial ring with $|x| = 1$ and $|y| = -1$. We have now proved our main theorem:
%

\begin{proof}[Proof of \cref{T:main_thm}]
	This is a direct consequence of \cref{L:generates,L:endoring}, using \cite[Propositions~3.2~and~3.3]{August/Cheung/Faber/Gratz/Schroll:2023}.
\end{proof}

\begin{remark}\label{R:explicit_equivalence}
	Choosing the generator $k[y]$ for $\smcm[\BZ]{k[x,y]/\langle x^2 \rangle}$, we can describe an explicit equivalence $F$ from  $\smcm[\BZ]{k[x,y]/\langle x^2 \rangle}$ to  $\smcm[\BZ]{k[x,y]/\langle x^2y \rangle}$ which sends
\begin{eqnarray*}
	C_i &\mapsto& k[y](2i) \; \; \text{for $i \in \BZ$,} \\
	D_i &\mapsto& k[y](2i+1) \; \;  \text{for $i \in \BZ$,} \\
	E_{i,j} &\mapsto & \langle x, y^{2(j-i)-1} \rangle (2i)  \; \; \text{for $j \geq i+1$,} \\
	F_{i,j} &\mapsto & \langle x, y^{2(j-i) -1} \rangle (2i+1) \; \; \text{for $j \geq i+1$,}\\
	G_{i,j} &\mapsto & \langle x, y^{2(j-i)-2} \rangle(2i+1) \; \; \text{for $j \geq i+2$,} \\
	H_{i,j} & \mapsto & \langle x, y^{2(j-i)} \rangle (2i) \; \; \text{for $j  \geq i+1$}.
\end{eqnarray*}
In terms of the combinatorial model described in \cite[Section~4]{August/Cheung/Faber/Gratz/Schroll:2023}, the objects $C_i$ and $D_i$ correspond to infinite arcs in the completed $\infty$-gon, the former ending in even labels and the latter in odd labels, the object $E_{i,j}$ and $F_{i,j}$ correspond to arcs of even length, the former ending in even labels, and the latter in odd labels, and the object $G_{i,j}$ and $H_{i,j}$ correspond to arcs of odd length. In particular, for any object $M$ in $\smcm[\BZ]{k[x,y]/\langle x^2 \rangle}$ we have 
\[
	F(M(1)) = F(M)(2)\,.
\]
The arcs are depicted in the following diagram:
\begin{center}
\includegraphics[scale=0.8]{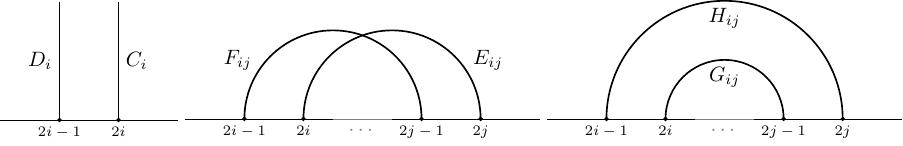}
\end{center}
\end{remark}

\begin{remark}
Note that the equivalence that we construct does not preserve the grading. 
While this follows from Remark \ref{R:explicit_equivalence}, there is also a conceptual reason: Take the respective orbit categories by the grading shift functor. If the equivalence preserved the grading, then the equivalence would descend to the orbit categories. Each of the orbit categories is equivalent to its respective non-graded singularity category by \cite[Proposition A.7]{Keller/Murfet/VandenBergh:2011} and the fact that all maximal Cohen-Macaulay modules over $A_\infty$- and $D_\infty$-singularities are graded. However, the non-graded singularity categories are not equivalent by e.g. \cite[p.23,p.25]{Schreyer:1987}; see also \cite[Theorem 5.3]{Matsui:2019}.
\end{remark}

\appendix

\section{Stable Hom spaces}

Similar to the computations of the morphisms between $C_i$ and $D_i$ in $\smcm[\BZ]{R}$ in the proof of \cref{L:endoring}, one can compute the morphisms between any two (non-trivial) indecomposable graded MCM-modules. The computations involving $E_{i,j}$, $F_{i,j}$, $G_{i,j}$ or $H_{i,j}$ are more involved, as the modules are bigger. We provide the explicit calculations for $\Hom{\smcm[\BZ]{R}}{E_{i,j}}{G_{k,\ell}}$ and summarize the results in the remaining cases.

For the morphisms with source $E_{i,j}$ with $j > i$ we consider the diagram
\begin{equation*}
\begin{tikzcd}[ampersand replacement=\&,column sep=large]
\CR(E_{i,j}) \&[-4em] : \&[-3em] \cdots \ar[r] \&[-1em] {\begin{matrix} R(i) \\ R(j) \end{matrix}} \ar[r,"{\left(\begin{smallmatrix} xy & y^{j-i} \\ 0 & -xy \end{smallmatrix}\right)}"] \& {\begin{matrix} R(i) \\ R(j) \end{matrix}} \ar[d] \ar[r,"{\left(\begin{smallmatrix} x & y^{j-i-1} \\ 0 & -x \end{smallmatrix}\right)}"] \& {\begin{matrix} R(i+1) \\ R(j+1) \end{matrix}} \ar[r] \&[-1em] \cdots \nospacepunct{.} \\
\&\&\&\& R(m)
\end{tikzcd}
\end{equation*}
This yields the following cocycles with associated homotopies:
\begin{enumerate}[a)]
\item $\begin{pmatrix} 0 & x^{m-j} \end{pmatrix} = \partial \begin{pmatrix} 0 & -x^{m-j-1} \end{pmatrix}$ when $m > j+1$; 
\item $\begin{pmatrix} 0 & x \end{pmatrix} = \partial \begin{pmatrix} 0 & -1 \end{pmatrix}$ when $m = j+1$; 
\item $\begin{pmatrix} 0 & xy^{j-m+1} \end{pmatrix} = \partial \begin{pmatrix} 0 & -y^{j-m+1} \end{pmatrix}$ when $j+1 > m \geq i+2$; 
\item $\begin{pmatrix} 0 & xy^{j-m+1} \end{pmatrix} = \partial \begin{pmatrix} 0 & -y^{j-m+1} \end{pmatrix} = \partial \begin{pmatrix} xy^{i-m+2} & 0 \end{pmatrix}$ when $i+2 > m$; 
\item $\begin{pmatrix} x^2 & 0 \end{pmatrix} = \partial \begin{pmatrix} x & y^{j-i-1} \end{pmatrix}$ when $m = i+2$; 
\item $\begin{pmatrix} x^{m-i} & 0 \end{pmatrix} = \partial \begin{pmatrix} x^{m-i-1} & 0 \end{pmatrix}$ when $m > i+2$; 
\item $\begin{pmatrix} xy^{i-m+1} & y^{j-m} \end{pmatrix} = \partial \begin{pmatrix} y^{i-m+1} & 0 \end{pmatrix}$ when $i+1 \geq m$. 
\end{enumerate}

We compute the morphisms with target $G_{k,\ell} = \langle x,y^{\ell-k-1}\rangle(k+1) \subseteq R(k)$ with $\ell > k+1$; that is $m = k+1$. We obtain non-zero morphisms for:
\begin{enumerate}[a)]
\item $k+1 > j+1$ and false;
\item $k+1 = j+1$ and true;
\item $j+1 > k+1 \geq i+2$ and $j-k-1+1 < \ell-k-1$;
\item $i+2 > k+1$ and $j-k-1+1 < \ell-k-1$ and false; 
\item $k+1 = i+2$ and $j-i-1 < \ell-k-1$; 
\item $k+1 > i+2$ and false;
\item $i+1 \geq k+1$ and $j-k-1 \geq \ell-k-1$ and $i-k-1+1 < \ell-k-1$.
\end{enumerate}
The case $j = k = i+1$ appears in b) and e), and the case $j > k = i+1$ appears in c) and e). However, the morphisms are the same up to sign in $\smcm[\BZ]{R}$ as
\begin{equation*}
\begin{pmatrix} 0 & x y^{j-i-1} \end{pmatrix} = - \begin{pmatrix} x^2 & 0 \end{pmatrix} + \partial \begin{pmatrix} x & 0 \end{pmatrix}\,.
\end{equation*}
This yields
\begin{equation*}
\dim_k \Hom{\smcm[\BZ]{R}}{E_{i,j}}{G_{k,\ell}} = \begin{cases}
1 & \ell-1 > j \geq k > i \\
1 & j \geq \ell > i+1 \geq k+1 \\
0 & \text{otherwise}\,.
\end{cases}
\end{equation*}

The computations in the other cases yield similar results; the space of morphisms is either 1-dimensional or zero. The conditions simplify when we consider
\begin{equation*}
\sExt[1]{R}{X}{Y} = \Hom{\smcm[\BZ]{R}}{\Omega X}{Y}\,.
\end{equation*}
We summarize the results in \cref{sExt_conditions}.

\begin{figure}
\centering
\includegraphics[scale=0.85]{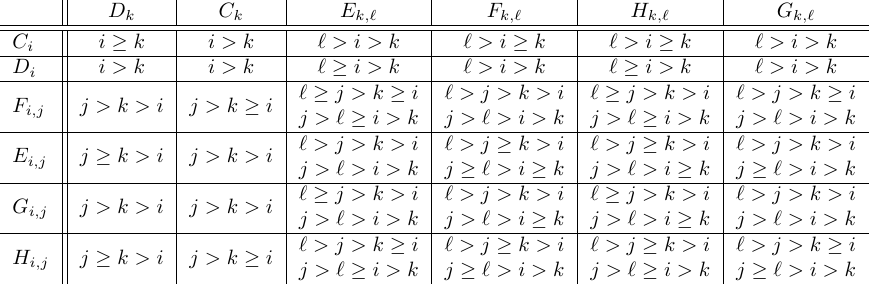}
\caption{The table provides the condition when $\Hom{\smcm[\BZ]{R}}{\Omega X}{Y}$ is 1-dimensional over $k$. In all other cases the set of morphisms in $\smcm[\BZ]{R}$ is zero. }
\label{sExt_conditions}
\end{figure}

\bibliographystyle{amsalpha}
\bibliography{references}

\end{document}